\documentclass[10pt,oneside]{amsart}
    
\usepackage{amsfonts}
\usepackage{amsthm}
\usepackage{amssymb}
\usepackage{amsmath}
\usepackage{multicol}
\usepackage{enumitem}
\usepackage{tikz}
\usepackage{mathtools}
\usepackage{graphicx}
\graphicspath{ {documents/} }
\usetikzlibrary{matrix}
\newcommand\mydots{\hbox to 1em{.\hss.\hss.}}
\newcommand{\bigslant}[2]{{\raisebox{.2em}{$#1$}\left/\raisebox{-.2em}{$#2$}\right.}}
\usepackage[all]{xy}
\theoremstyle{plain}
      \newtheorem{theorem}{Theorem}[section]
      \newtheorem{lemma}[theorem]{Lemma}
      \newtheorem{corollary}[theorem]{Corollary}
      \theoremstyle{definition}
      \newtheorem{definition}[theorem]{Definition}
      \theoremstyle{remark}
      \newtheorem{remark}[theorem]{Remark}
      \theoremstyle{proposition}
      \newtheorem{proposition}[theorem]{Proposition}
      \theoremstyle{definition}
\newtheorem{example}{Example}[section]
      
      \newcommand\blfootnote[1]{%
  \begingroup
  \renewcommand\thefootnote{}\footnote{#1}%
  \addtocounter{footnote}{-1}%
  \endgroup
}

      \newcommand{\C}{{\mathbb{C}}}

   		\newcommand{\opn}{\operatorname}
   		
      \makeatletter
      \def\@setcopyright{}
      \def\serieslogo@{}
      \makeatother
      \title{On the rationality of the singularities of the \textit{A}\textsubscript{2}-loci}
      \author{Natalia Kolokolnikova}
       \address{University of Geneva, Switzerland}
   \email{Natalia.Kolokolnikova@unige.ch}
\begin{document}
\maketitle
\section{Introduction}
\blfootnote{The research was supported by Grant 156645 of the Swiss Science Foundation} In the last twenty years, a significant progress has been made in calculating Thom polynomials of contact singularities, see \cite{Rim}, \cite{BSz}, \cite{FRim}, and \cite{Kaz}. We note that the residue formulas for $A_n$-singularities obtained in \cite{BSz} are reminiscent of the Jeffrey-Kirwan residue for reductive quotients \cite{JKir}. According to Boutot \cite{Bou}, if a variety has rational singularities, then so does its quotient by the action of a reductive group. Thus, the natural question is whether the  $A_n$-loci can be presented as a reductive quotient and, in particular, if they have rational singularities. The same question appears in the recent work of Rim\'anyi and Szenes \cite{RSz} on the $K$-theoretic invariants of the same loci. In this paper we show that, in general, the $A_2$-loci have singularities worse than rational, and therefore they can not be presented as a GIT quotient of a smooth variety with respect to a reductive group. \par
We begin with recalling some facts about smooth resolutions and a brief introduction to singularity theory.\par
Let $X$ be an affine variety. If $Y$ is smooth and there exists a proper birational map $f\colon Y \rightarrow X,$ then we say that $Y$ is a \textit{smooth resolution} of $X.$
\begin{proposition}
The cohomology groups $H^i(Y,\mathcal{O}_Y)$ do not depend on the smooth resolution $Y,$ i.e. are invariants of $X$.\end{proposition}
This fact follows from the Elkik-Fujita Vanishing Theorem \cite{Kaw}.\par
\begin{proposition} \label{pr}
$H^0(X,\mathcal{O}_X)=H^0(Y,\mathcal{O}_Y)$ if and only if $X$ is normal. 	
\end{proposition}
If $X$ is not normal, there exists a unique \textit{normalisation} of $X$ -- normal affine variety $\widetilde{X}.$ In this case $H^0(\widetilde{X},\mathcal{O}_{\widetilde{X}})=H^0(Y,\mathcal{O}_Y),$ but $H^0(\widetilde{X},\mathcal{O}_{\widetilde{X}}) \neq H^0(X,\mathcal{O}_X).$ The proof of the proposition above is based on the universal property of the normalization and Zariski's Main Theorem \cite{Mum}.
\begin{definition}
Let $X$ be a normal affine variety, then $X$ has \textit{rational singularities} if $H^i(Y,\mathcal{O}_Y)=0$ for all $i>0.$
\end{definition}
Suppose a compact Lie group $G$ acts on the affine space $\mathbb{A}^N.$ Let $X\subset \mathbb{A}^N$ be a $G$-invariant subvariety. $Y$ is called an \textit{equivariant} smooth resolution of $X$ if $Y$ is smooth, $G$ acts on $Y,$ and the map $f\colon Y\rightarrow X$ is proper birational and $G$-equivariant.\par
Let $T$ be the maximal torus of $G.$ One of the natural questions that arises in \cite{RSz} is whether $\chi[H^0(X,\mathcal{O}_X)](t)$ is equal to $\chi\left[\sum(-1)^i H^i (Y,\mathcal{O}_Y)\right](t),$ $t \in T.$ Note that while $X$ is an affine variety and therefore $H^i(X,\mathcal{O}_X)=0$ for $i>0,$ this in not necessarily true for $H^i(Y,\mathcal{O}_Y).$
\begin{proposition}
	Let $G$ be a compact Lie group acting on $\mathbb{A}^N.$ Let $X \subset \mathbb{A}^N$ be a $G$-invariant subvariety, and let $Y$ be its smooth $G$-equivariant resolution. Let $T$ be the maximal torus of $G.$
The equality	
$$\chi[H^0(X,\mathcal{O}_X)](t) = \chi\left[\sum(-1)^i H^i (Y,\mathcal{O}_Y)\right](t),\ t\in T$$ 
holds if and only if $X$ has rational singularities.
\end{proposition}
In this paper we study whether certain singularity loci have rational singularities. To give the definition of the main object of this paper, the $A_2$-locus, we recall the necessary notions of singularity theory. For a  more detailed introduction see \cite{BSz} or \cite{Kol}.\par
Denote by $x_1,\dots,x_n$ coordinates on $\mathbb{C}^n.$ We introduce the notation $$J(n)=\{h \in \mathbb{C}[[x_1,\mydots,x_n]]\ |\ h(0)=0\}$$ for the algebra of power series without constant term, $\langle \overline{\text{x}}^{d+1} \rangle$ for the ideal generated by monomials in $x_1,\mydots, x_n$ of degree $d+1$ and $J_d(n)=J(n)/\langle \overline{\text{x}}^{d+1}\rangle$ for the space of $d$-jets of holomorphic functions near the origin.\par
Let $J_d(n,k)$ be the space of $d$-jets of holomorphic maps $(\mathbb{C}^n,0)\rightarrow (\mathbb{C}^k,0):$
$$J_d(n,k)=\operatorname{Hom}(\mathbb{C}^k, J_d(n)).$$\par
An element of this space can be thought of as a $k$-tuple of elements of $J_d(n):$ 
$$J_d(n,k)\cong\{(P_1,\mydots,P_k)\ |\ P_i\in J_d(n)\}.$$\par
$J_d(n,k)$ is a finite-dimensional complex vector space equipped with $\operatorname{Gl}(n)\times \operatorname{Gl}(k)$-action. In this paper we will consider $n\leq k.$ \par
We will call an algebra $N$ \textit{nilpotent} if it is finite dimensional and if there exists a natural number $m$ such that the product of each $m$ elements of the algebra vanishes, that is, $N^m=0.$ $J_d(n)$ is nilpotent: $(J_d(n))^{d+1}=0,$ the algebra $J_d(1)$ is often denoted by $A_d= t\mathbb{C}[t]/t^{d+1}.$
\begin{definition} An algebra $C$ is $(1,1,\mydots,1)$-\textit{filtered} if $C$ has an increasing finite sequence of subspaces $\{0\}\subset F_m\subset F_{m-1}\subset \mydots \subset F_{1} \subset F_0=C$ such that $F_i\cdot F_j \subset F_{i+j}$ and $\operatorname{dim}F_i/F_{i+1}=1.$
\end{definition}
Nilpotent algebras have a natural filtration: $\{0\}\subset N^{m-1}\subset N^{m-2}\subset \mydots \subset N^2\subset N.$ In case of $A_d$, this filtration is a $(1,1,\mydots,1)$-filtration. 
\begin{definition}
	$A_d$-\textit{singularity locus} is given by $$\Theta_{A_d}^{n,k}=\overline{{{\{(P_1,\dots,P_k)\in J_d(n,k)\ |\ J_d(n)/\langle P_1,\dots,P_k\rangle \cong A_d\}}}}.$$
\end{definition}
$\Theta_{A_d}^{n,k}$ is a $\operatorname{Gl}(n)\times \operatorname{Gl}(k)$-invariant affine subvariety in $J_d(n,k).$\par
This paper is devoted to the study of the rationality of the singularities of $\widetilde{\Theta_{A_2}^{n,k}}.$ \par
Let us briefly look at a simpler case, the $A_1$-locus:
$$\Theta_{A_1}^{n,k}=\overline{\{M\in \operatorname{Hom}(\mathbb{C}^n,\mathbb{C}^k)\ |\ \operatorname{rk}M<n\}},$$
i.e. for every $M\in \Theta_{A_1}^{n,k}$ there exists a non-zero eigenvector $v\in \mathbb{C}^n$ such that $Mv=0.$
\begin{proposition}
The space 
$$\{(M,v)\ |\ Mv=0,\ M\in \operatorname{Hom}(\mathbb{C}^n,\mathbb{C}^k),\ v\in \mathbb{C}^{n}\}\subset \operatorname{Hom}(\mathbb{C}^n,\mathbb{C}^k)\times \mathbb{P}^{n-1}$$
is an equivariant smooth resolution of $\Theta_{A_1}^{n,k}.$ 
\end{proposition}
This space can be understood as follows: let us fix an element $v\in \mathbb{P}^{n-1}$ and describe the set $\{M \in \operatorname{Hom}(\mathbb{C}^n,\mathbb{C}^k)\ |\ Mv=0\}.$\par
There is a tautological sequence of vector bundles on $\mathbb{P}^{n-1}:$
 \begin{center}
\begin{tikzpicture}
  \matrix (m) [matrix of math nodes,row sep=3em,column sep=4em,minimum width=2em]
  {
     \mathcal{O}(-1)=L & \mathbb{C}^n & Q \\
     & \mathbb{P}^{n-1}  & \\};
  \path[-stealth]
    (m-1-1) edge node {} (m-1-2)
    (m-1-2) edge node {} (m-1-3)
    (m-1-2) edge node {} (m-2-2)
            ;
\end{tikzpicture}\\
\end{center}
We can apply $\operatorname{Hom}(*,\mathbb{C}^k)$ to it and obtain the following sequence:
\begin{center}
\begin{tikzpicture}
  \matrix (m) [matrix of math nodes,row sep=3em,column sep=4em,minimum width=2em]
  {
     \operatorname{Hom}(Q,\mathbb{C}^k) & \operatorname{Hom}(\mathbb{C}^n,\mathbb{C}^k) & \operatorname{Hom}(L,\mathbb{C}^k) \\
     & \mathbb{P}^{n-1} & \\};
  \path[-stealth]
    (m-1-1) edge node {} (m-1-2)
    (m-1-2) edge node {} (m-1-3)
    (m-1-2) edge node {} (m-2-2)
            ;
\end{tikzpicture}\\
\end{center}
The map $\operatorname{Hom}(\mathbb{C}^n,\mathbb{C}^k) \rightarrow \operatorname{Hom}(L,\mathbb{C}^k)$ can be interpreted as the evaluation map $M\mapsto Mv$ for a fixed $v\in \mathbb{P}^{n-1}.$ Its kernel is exactly $\operatorname{Hom}(Q,\mathbb{C}^k).$\par
The equivariant smooth resolution of $\Theta_{A_1}^{n,k}$ defined above may be presented as the following vector bundle:
\begin{center}
\begin{tikzpicture}
  \matrix (m) [matrix of math nodes,row sep=3em,column sep=4em,minimum width=2em]
  {
     \operatorname{Hom}(Q,\mathbb{C}^k) & \Theta_{A_1}^{n,k} \\
      \mathbb{P}^{n-1} & \\};
  \path[-stealth]
    (m-1-1) edge node {} (m-2-1)
    (m-1-1) edge node {} (m-1-2)
            ;
\end{tikzpicture}\\
\end{center}\par
It is well-known that $\Theta_{A_1}^{n,k}$ has rational singularities. In this paper we study the rationality of the singularities of $\Theta_{A_2}^{n,k}$ and prove the following theorems.
\begin{theorem}
$\widetilde{\Theta_{A_2}^{n,k}}$ in general can have singularities worse than rational.
\end{theorem}
\begin{theorem}
$\widetilde{\Theta_{A_2}^{n,n}}$ has rational singularities.	
\end{theorem}
Before proving the main theorems, we recall the explicit construction for the equivariant smooth resolution of $\Theta_{A_2}^{n,k},$ the Borel-Weil-Bott theorem, and demonstrate the spectral sequences technique that will allow us to study the rationality of the singularities of the $A_2$-loci.\\
\noindent \textbf{Acknowledgements.} I would like to thank my thesis advisor Andr\'as Szenes for his guidance, and Rich\'ard Rim\'anyi, Anton Fonarev, Maxim Kazarian and S\'andor Kov\'acs for useful discussions.
\section{Preliminaries}  
\subsection{Equivariant smooth resolution of the $\bold{A_2}$-locus}
 In this section we recall an explicit construction for the equivariant smooth resolution of the $A_2$-locus following \cite{Kaz}. The general case is discussed in \cite{Kaz} and \cite{BSz}.\par
 Before we present the equivariant smooth resolution of $\Theta_{A_2}^{n,k},$ we need to introduce some preliminary notions.
 \begin{definition}
	The \textit{curvilinear Hilbert scheme} of order 2 is defined as follows: $$\opn{Hilb}_{A_2}(\mathbb{C}^n)\cong\overline{\{I \subset J_2(n)\ |\ J_2(n)/I\cong A_2\}}.$$
\end{definition}
Each ideal $I\in \opn{Hilb}_{A_2}(\mathbb{C}^n)$ comes with the tautological sequence:\\
\begin{center}
\begin{tikzpicture}
  \matrix (m) [matrix of math nodes,row sep=3em,column sep=4em,minimum width=2em]
  {
     I & J_2(n) & N\cong J_2(n)/I \\};
  \path[-stealth]
    (m-1-1) edge node {} (m-1-2)
    (m-1-2) edge node {} (m-1-3)
            ;
\end{tikzpicture}\\
\end{center}
To construct a smooth equivariant resolution of $\Theta_{A_2}^{n,k}$ we start with the following vector bundle: \\
\begin{center}
\begin{tikzpicture}
  \matrix (m) [matrix of math nodes,row sep=3em,column sep=4em,minimum width=2em]
  {
     \operatorname{Hom}(\mathbb{C}^k,I) & \Theta_{A_2}^{n,k} \\
     \opn{Hilb}_{A_2}(\mathbb{C}^n) \\};
  \path[-stealth]
    (m-1-1) edge node {} (m-1-2)
    (m-1-1) edge node {} (m-2-1)
            ;
\end{tikzpicture}\\
\end{center}
The fiber over $I \in \opn{Hilb}_{A_2}(\mathbb{C}^n)$ is the space of all $k$-tuples of elements of $I.$ The set of $k$-tuples of elements of $I$ that generate $I$ is Zariski open in $\operatorname{Hom}(\mathbb{C}^k,I)$ and the projection $\operatorname{Hom}(\mathbb{C}^k,I) \twoheadrightarrow J_d(n,k)\supset \Theta_{A_2}^{n,k}$ is proper.\par
This vector bundle is not a smooth equivariant resolution of $\Theta_{A_2}^{n,k}$ because $\opn{Hilb}_{A_2}(\mathbb{C}^n)$ is not smooth. The next step is to find a smooth equivariant resolution of $\opn{Hilb}_{A_2}(\mathbb{C}^n)$.\par
Since every $I\in \opn{Hilb}_{A_2}(\mathbb{C}^n)$ is equipped with the tautological sequence mentioned above, we can rewrite $\opn{Hilb}_{A_2}(\mathbb{C}^n)$ as
$$
\bigslant{\opn{Hilb}_{A_2}(\mathbb{C}^n)={\{f\colon J_2(n)\rightarrow N\ |\ \operatorname{dim}N=2,\ f\ \text{-- surj. alg. homomorphism}}\}}{\sim}
$$
The equivalence relation is defined as follows: $f\sim f'$ if the diagram commutes:
\begin{center}
\begin{tikzpicture}
  \matrix (m) [matrix of math nodes,row sep=3em,column sep=4em,minimum width=2em]
  {
    &  N\\
     J_2(n,k)\\
      & N\\};
  \path[-stealth]
    (m-2-1) edge node[above] {$f$} (m-1-2)
    (m-2-1) edge node [below]{$f'$} (m-3-2)
    (m-1-2) edge node [right]{$\cong$} (m-3-2)
    (m-3-2) edge node {} (m-1-2)
            ;
\end{tikzpicture}\\
\end{center}

We will be interested in $(1,1)$-filtered $2$-dimensional nilpotent algebras. There are two different types of them: 
\begin{itemize}
	\item $A_2$ with the natural $(1,1)$-filtration: $A_2^2\subset A_2,$
	\item algebra $N$ generated by two elements, such that the product of any two elements of $N$ is 0. This algebra does not have a natural $(1,1)$-filtration, so we introduce an artificial $(1,1)$-filtration $F_1\subset N,$ where $F_1$ is any line in $N.$ 
	\end{itemize}\par
	Let us introduce the notation for filtered algebra homomorphisms. Suppose  filtered algebras $N$ and $C$. We will denote a homomorphism compatible with the filtrations on $N$ and $C$ by
	$$f\colon N \xrightarrow[]{\Delta} C$$
\begin{proposition}
	The smooth equivariant resolution of $\opn{Hilb}_{A_2}(\mathbb{C}^n)$ is given by
	\begin{equation*}
		\widehat{\opn{Hilb}}_{A_2}(\mathbb{C}^n)=\bigslant{{\{f\colon J_2(n)\xrightarrow[]{\Delta} N\ |\ N\ \text{--}\ 2\text{-dim.}\ (1,1)\text{-filt.,}\ f\text{ -- surj.}\}}}{\sim},
	\end{equation*}
	The equivalence is taken up to a filtered algebra isomorphism:
	\begin{center}
\begin{tikzpicture}
  \matrix (m) [matrix of math nodes,row sep=3em,column sep=4em,minimum width=2em]
  {
    &  N\\
     J_2(n,k)\\
      & N\\};
  \path[-stealth]
    (m-2-1) edge node[above] {$\Delta$} (m-1-2)
    (m-2-1) edge node [below]{$\Delta$} (m-3-2)
    (m-1-2) edge node [right]{$\Delta$} (m-3-2)
    (m-3-2) edge node {} (m-1-2)
            ;
\end{tikzpicture}\\
\end{center}
		\end{proposition}
The following vector bundle is a smooth equivariant resolution of the $A_2$-locus:
\begin{center}
\begin{tikzpicture}
  \matrix (m) [matrix of math nodes,row sep=3em,column sep=4em,minimum width=2em]
  {
     \operatorname{Hom}(\mathbb{C}^k,I) & \Theta_{A_2}^{n,k} \\
     \widehat{\opn{Hilb}}_{A_2}(\mathbb{C}^n) \\};
  \path[-stealth]
    (m-1-1) edge node {} (m-1-2)
    (m-1-1) edge node {} (m-2-1)
            ;
\end{tikzpicture}\\
\end{center}\par
	Now we need to find a simpler interpretation of this resolution.\par	
Let $g$ be the inverse of the canonical map $ J_2(n)\rightarrow J_2(n)/ {\left(J_2(n)\right)}^2\cong \mathbb{C}^n$:
$$g \colon \mathbb{C}^n\rightarrow J_2(n)$$
Let us denote its image by $Im(g)=E^*.$ $E^*$ is the linear part of $J_2(n).$\par
Let $A^{\Delta}$ be a $2$-dimensional algebra equipped with the $(1,1)$-filtration and $f\in \widehat{\opn{Hilb}}_{A_2}(\mathbb{C}^n).$
We can define two natural maps
$$\psi_1\colon E\rightarrow A^{\Delta},\ \psi_1=f\big|_{E^*}$$
$$\psi_2\colon \operatorname{Sym}^2 A^{\Delta}\rightarrow A^{\Delta}$$
\begin{proposition}
The linear map $\psi_1\oplus \psi_2\colon E^*\oplus \operatorname{Sym}^2 A^{\Delta}\rightarrow A^{\Delta}$ is surjective.	
\end{proposition}
\begin{proposition}
	Let $N$ be a $2$-dimensional filtered vector space.\par
	$\widehat{\opn{Hilb}}_{A_2}(\mathbb{C}^n)$ is in one-to-one correspondence with the set of isomorphism classes of pairs $(\psi_1,\psi_2),$ where $\psi_2\colon \operatorname{Sym}^2 N\rightarrow N$ is a map giving $N$ an associative commutative algebra structure and $\psi_1\colon \left(\mathbb{C}^n\right)^*\rightarrow N$ is a linear map such that $\psi_1\oplus \psi_2$ is surjective. Pairs $(\psi_1,\psi_2)$ are taken up to filtered algebra isomorphism. 
\end{proposition}
Let us describe $\widehat{\opn{Hilb}}_{A_2}(\mathbb{C}^n)$ using this correspondence.\par
Suppose $N$ be a $2$-dimensional vector space with a  filtration $N_2 \subset N,$ where $N_2$ is a line in $N$. 
\begin{center}
\begin{tikzpicture}
  \matrix (m) [matrix of math nodes,row sep=3em,column sep=4em,minimum width=2em]
  {
     (\mathbb{C}^n)^* & & N/N_2 \\
     & N & \\};
  \path[-stealth]
    (m-1-1) edge node [above]{$\psi_1^{'}$} (m-1-3) 
    (m-1-1) edge node [below] {$\psi_1$} (m-2-2)
    (m-2-2) edge node {} (m-1-3)
            ;
\end{tikzpicture}\\
\end{center}
The kernel of this map is defined by $Ker (\psi_1^{'})=\{V \subset (\mathbb{C}^n)^*\ |\ \operatorname{dim}V=n-1\}=\mathbb{P}^{n-1}(\mathbb{C}^n)^*\cong \mathbb{P}^{n-1}.$ Let us denote $\mathcal{O}(-1)$ over $\mathbb{P}^{n-1}$ by $L_1$ and the quotient bundle by $Q_1.$\par
The kernel of $\psi_1\oplus \psi_2$ is then a codimension $2$ subspace in $\operatorname{Sym}^2 L_1 \oplus (\mathbb{C}^n)^*\cong L_1^2\oplus (\mathbb{C}^n)^*,$ such that it's projection is of codimension 1 in $(\mathbb{C}^n)^*,$ that is:
\begin{center}
\begin{tikzpicture}
  \matrix (m) [matrix of math nodes,row sep=3em,column sep=4em,minimum width=2em]
  {
     \mathbb{P}^{n-1}(Q_1^*\oplus (L_1^*)^2) & \mathbb{P}(Q_1\oplus L_1^2) \\
     & \mathbb{P}^{n-1} \\};
  \path[-stealth]
    (m-1-1) edge node [above] {$\cong$} (m-1-2) 
    (m-1-2) edge node {} (m-2-2)
            ;
\end{tikzpicture}\\
\end{center}\par
Let us fix a point $a$ in $\mathbb{P}^{n-1}.$ The fiber over this point is $\mathbb{P}((Q_1 \oplus L_1^2)|_a)=\mathbb{P}V_a.$ Let $V$ be an $n$-dimensional complex vector space. We have the following tautological sequence on $\mathbb{P}V_a:$
 \begin{center}
\begin{tikzpicture}
  \matrix (m) [matrix of math nodes,row sep=3em,column sep=4em,minimum width=2em]
  {
     \mathcal{O}(-1)=L_2 & V_a & Q_2 \\
     & \mathbb{P}V_a  & \\};
  \path[-stealth]
    (m-1-1) edge node {} (m-1-2)
    (m-1-2) edge node {} (m-1-3)
    (m-1-2) edge node {} (m-2-2)
            ;
\end{tikzpicture}\\
\end{center}\par
This description allows us to present the smooth equivariant resolution of the $A_2$-locus in the following form:
\begin{center}
\begin{tikzpicture}
  \matrix (m) [matrix of math nodes,row sep=3em,column sep=4em,minimum width=2em]
  {
     \operatorname{Hom}\left(\mathbb{C}^k,{{\operatorname{Sym}^2 \mathbb{C}^n \oplus Q_1}\over{L_2}}\right) & \Theta_{A_2}^{n,k} \\
     \mathbb{P}(Q_1\oplus L_1^2) \\
     \mathbb{P}^{n-1}\\};
  \path[-stealth]
    (m-1-1) edge node {} (m-1-2)
    (m-1-1) edge node {} (m-2-1)
    (m-2-1) edge node {} (m-3-1)
            ;
\end{tikzpicture}\\
\end{center}
\subsection{The Borel-Weil-Bott theorem}
Let $V$ be an $n$-dimensional complex vector space. In this paper we use the Borel-Weil-Bott theorem to compute the cohomology of $\operatorname{Gl}(V)$-equivariant vector bundles on $\mathbb{P}V.$\par
The irreducible representations of $\operatorname{Gl}(V)$ are parametrized by their highest weights -- non-increasing integer partitions $\lambda$ of length $n$ (we allow the entries to be equal to 0): $\lambda_1\geq \lambda_2\geq \lambda_n\geq 0$. We will denote the irreducible representation of $\operatorname{Gl}(V)$ of highest weight $\lambda$ by $\Sigma^{\lambda}V.$\par
Consider the canonical sequence of vector bundles on $\mathbb{P}V:$ 
\begin{center}
\begin{tikzpicture}
  \matrix (m) [matrix of math nodes,row sep=3em,column sep=4em,minimum width=2em]
  {
     \mathcal{O}(-1)=L & V & Q \\
     & \mathbb{P}V  & \\};
  \path[-stealth]
    (m-1-1) edge node {} (m-1-2)
    (m-1-2) edge node {} (m-1-3)
    (m-1-2) edge node {} (m-2-2)
            ;
\end{tikzpicture}\\
\end{center}\par
 We will be interested in computing the cohomology of $\opn{Gl}(V)$-equivariant vector bundles of the form $\Sigma^{\lambda}Q\otimes L^m$ on $\mathbb{P}V$.  Following the argument in \cite{Fon}, a vector bundle of this form may be presented as a pushforward of the corresponding line bundle on the flag variety of $\opn{Gl}(V).$ Thus, we may compute its cohomology using the following interpretation of the Borel-Weil-Bott theorem.\par
\begin{theorem}[The Borel-Weil-Bott theorem, \cite{Fon}]
Consider an irreducible $\opn{Gl}(V)$-equivariant vector bundle $\Sigma^{\lambda}Q\otimes L^m$ on $\mathbb{P}V.$ Denote by $(\lambda,m)$ the concatenation of $\lambda=(\lambda_1,\mydots,\lambda_{n-1})$ and $m,$ and by $\rho=(n,n-1,\dots,1)$ the half-sum of the positive roots of $\operatorname{Gl}(V).$\par
 Consider $(\lambda,m)+\rho=(\lambda_1 + n, \lambda_2 + n-1,\mydots, \lambda_{n-1}+2, m+1).$\par
  If two entries of $(\lambda,m)+\rho$ are equal, then $$H^i(\mathbb{P}V,\Sigma^{\lambda}Q\otimes L^m)=0\text{ for all } i.$$ \par
 If all entries of $(\lambda,m)+\rho$ are distinct, then there exists a unique permutation $\sigma$ such that $\sigma((\lambda,m)+\rho)$ is strictly decreasing, i.e. dominant. The length of this permutation, $l(\sigma),$ is the number of strictly increasing pairs of elements of $(\lambda,m)+\rho$.\par Then
$H^i(\mathbb{P}V,\Sigma^{\lambda}Q\otimes L^m)=$
$\begin{cases} \Sigma^{\sigma((\lambda,m)+\rho)-\rho}V\text{ if }i=l(\sigma)\\
0\text{ otherwise.}
\end{cases}$
\end{theorem}
	\begin{example} \label{example} Let us compute $H^i(\mathbb{P}^3, Q\otimes \operatorname{Sym}^2 Q\otimes L^5).$\par
	First, we need to decompose $Q\otimes \operatorname{Sym}^2 Q$ into the direct sum of irreducible representations. The algorithm is the same as in decomposing the product of two corresponding Schur polynomials into a sum of Schur polynomials, for the details see \cite{Ful} or \cite{Ful2}.\par In the case of $Q\otimes \operatorname{Sym}^2 Q$, we obtain the following:
	$$Q\otimes \operatorname{Sym}^2 Q=\Sigma^{(1,0,0)}Q\otimes \Sigma^{(2,0,0)}Q=\Sigma^{(3,0,0)}Q+\Sigma^{(2,1,0)}Q.$$\par
	To compute the cohomology groups of the initial sheaf, we compute the cohomology groups of both irreducible summands: $$H^i(\mathbb{P}^3, Q\otimes \operatorname{Sym}^2 Q\otimes L^5)=H^i(\mathbb{P}^3, \Sigma^{(3,0,0)}Q\otimes L^5)\oplus H^i(\mathbb{P}^3, \Sigma^{(2,1,0)}Q\otimes L^5).$$\par
	Applying the Borel-Weil-Bott theorem to $\Sigma^{(3,0,0)}Q\otimes L^5,$ we first construct the sequence $(\lambda,m)$: here $\lambda=(3,0,0)$ and $m=5.$ We see that $(\lambda,m)+\rho=(3,0,0,5)+(4,3,2,1)=(7,3,2,6)$ has no repetitions. The unique permutation making $(7,3,2,6)$ decreasing is $\sigma=(2,3,4)$. Since there are two increasing pairs in $(7,3,2,6),$ namely, $\{3,6\}$ and $\{2,6\},$ $l(\sigma)$ -- the length of $\sigma$ --  is $2.$ Finally, $\sigma((\lambda,m)+\rho)-\rho=(7,6,3,2)-(4,3,2,1)=(3,3,1,1),$ so the only non-zero cohomology group is $$H^2(\mathbb{P}^3, \Sigma^{(3,0,0)}Q\otimes L^5)=\Sigma^{(3,3,1,1)}\mathbb{C}^4.$$\par
	The second irreducible summand is $\Sigma^{(2,1,0)}Q\otimes L^5.$ Here we obtain $(\lambda,m)+\rho=(2,1,0,5)+(4,3,2,1)=(6,4,2,6)$ -- there are repetitions, so $$H^i(\mathbb{P}^3,\Sigma^{(2,1,0)}Q\otimes L^5)=0\ \text{for all}\ i.$$\par
	The final answer is $H^i(\mathbb{P}^3, Q\otimes \operatorname{Sym}^2 Q\otimes L^5)=$
	$\begin{cases}
	\Sigma^{(3,3,1,1)}\mathbb{C}^4\ \text{if}\ i=2\\
		0\ \text{if}\ i \neq 2
		\end{cases}.$
		\end{example}
\section{Main results}
In this section we show that $\widetilde{\Theta_{A_2}^{n,n}},$ the normalization of $\Theta_{A_2}^{n,n},$ has rational singularities, and give an example, where $\widetilde{\Theta_{A_2}^{n,k}}$ has singularities worse than rational.\par
Consider the quasi-projective variety $Y$ -- Kazarian's smooth resolution of $\Theta_{A_2}^{n,k}:$
\begin{center}
\begin{tikzpicture}
  \matrix (m) [matrix of math nodes,row sep=3em,column sep=4em,minimum width=2em]
  {
     Y = \operatorname{Hom}\left(\mathbb{C}^k,{{\operatorname{Sym}^2 \mathbb{C}^n \oplus Q_1}\over{L_2}}\right) & \Theta_{A_2}^{n,k} \\
     \mathbb{P}(Q_1\oplus L_1^2) \\
     \mathbb{P}^{n-1}\\};
  \path[-stealth]
    (m-1-1) edge node {} (m-1-2)
    (m-1-1) edge node [right]{$p_1$} (m-2-1)
    (m-2-1) edge node [right]{$p_2$} (m-3-1)
            ;
\end{tikzpicture}\\
\end{center}\par
By definition, $\widetilde{\Theta_{A_2}^{n,k}}$ has rational singularities if $H^i (Y,\mathcal{O}_Y)= 0$ for all $i>0.$ We will compute these cohomology groups step by step, by pushing forward along the tower. \par
Fix a point $a$ in $\mathbb{P}^{n-1},$ the fiber over this point is $p_2^{-1}(a)= \mathbb{P}((Q_1 \oplus L_1^2)|_{a}) \cong \mathbb{P}V_a,$ where $V_a$ is an $n$-dimensional complex vector space. Let us also denote the constant sheaf $(\opn{Sym}^2 \C^n\oplus Q_1)|_a$ on $\mathbb{P}V_a$ by $W.$\par
Since the fiber over a point $b$ in $\mathbb{P}V_a$, $\left(\operatorname{Hom}\left(\mathbb{C}^k,{{\operatorname{Sym}^2 \mathbb{C}^n \oplus Q_1}\over{L_2}}\right)\right){\bigg \rvert}_b$, is affine, we have $H^i (Y,\mathcal{O}_Y)=H^i (\mathbb{P}V_a,(p_1)_* \mathcal{O}_Y).$ Moreover, the $\C^*$-action on the fiber allows us to decompose $(p_1)_* \mathcal{O}_Y$ into homogeneous components:
$$(p_1)_* \mathcal{O}_Y=\mathcal{O}_Y |_{p_1^{-1}(b)}\cong \bigoplus_l \opn{Sym}^l \left( {{W\otimes \mathbb{C}^k}\over{L_2\otimes \mathbb{C}^k}}\right).$$
This decomposition leads to the following identity on the level of cohomology: $$H^i (Y,\mathcal{O}_Y)=H^i (\mathbb{P}V_a,(p_1)_* \mathcal{O}_Y)=\bigoplus_l H^i \left(\mathbb{P}V_a,\opn{Sym}^l \left( {{W\otimes \mathbb{C}^k}\over{L_2\otimes \mathbb{C}^k}}\right)\right).$$\par
Let us compute $H^i \left(\mathbb{P}V_a,\opn{Sym}^l \left( {{W\otimes \mathbb{C}^k}\over{L_2\otimes \mathbb{C}^k}}\right)\right).$ We start with the Koszul resolution \cite{Br} of $\operatorname{Sym}^l\left( {{W\otimes \mathbb{C}^k}\over{L_2\otimes \mathbb{C}^k}}\right)$: 
$$\Lambda^l(L_2\otimes\mathbb{C}^k)\rightarrow \Lambda^{l-1}(L_2\otimes\mathbb{C}^k)\otimes \operatorname{Sym}^1 (W\otimes \mathbb{C}^k)\rightarrow \dots $$ 
$$\dots \rightarrow \Lambda^{l-i}(L_2\otimes \mathbb{C}^k)\otimes \operatorname{Sym}^i(W\otimes \mathbb{C}^k)\rightarrow \dots $$
$$\dots \rightarrow \Lambda^1(L_2)\otimes \operatorname{Sym}^{l-1}(W\otimes\mathbb{C}^k) \rightarrow \operatorname{Sym}^l(W\otimes \mathbb{C}^k)\rightarrow \operatorname{Sym}^l\left( {{W\otimes \mathbb{C}^k}\over{L_2\otimes \mathbb{C}^k}}\right)$$\par
We are interested in the case when $l$ is sufficiently large. Note that since $L_2$ is a line bundle, $\Lambda^i(L_2 \otimes \mathbb{C}^k)$ vanishes for $i>k.$ Using these facts we can rewrite the resolution as follows.\par
\textbf{Resolution 1:}
$$L_2^k\otimes\Lambda^k(\mathbb{C}^k)\rightarrow \mydots \rightarrow L_2^{k-i}\otimes \Lambda^{l-i}(\mathbb{C}^k)\otimes \operatorname{Sym}^i(W\otimes \mathbb{C}^k)\rightarrow \mydots $$ $$\mydots \rightarrow \operatorname{Sym}^l(W\otimes \mathbb{C}^k)\rightarrow \operatorname{Sym}^l\left( {{W\otimes \mathbb{C}^k}\over{L_2\otimes \mathbb{C}^k}}\right)$$
According to the Borel-Weil-Bott theorem, 
\begin{itemize}
	\item $H^{n-1}(\mathbb{P}V_a,\ \mathcal{O}(-m))\cong \operatorname{Sym}^{m-n}V_a\otimes \operatorname{det}V_a$ if $m-n\geq 0$,
\item $H^{n-1}(\mathbb{P}V_a,\ \mathcal{O}(-m))\cong 0$ if $m-n<0$,
\item $H^{i}(\mathbb{P}V_a,\ \mathcal{O}(-m))\cong 0$ if $i\neq n-1.$
\end{itemize}\par
This knowledge allows us to write down the Leray spectral sequence, which is a collection of indexed pages, i.e. tables with arrows pointing in the direction $(n,n-1)$ on the $n$-th page. The Leray spectral sequence allows us to obtain the cohomology groups of $\operatorname{Sym}^l\left( {{W\otimes \mathbb{C}^k}\over{L_2\otimes \mathbb{C}^k}}\right)$ by computing successive approximations. On the first page of the Leray spectral sequence, to each sheaf in the resolution above corresponds a column of its cohomology groups: 
\begin{center}
	\includegraphics[width=\textwidth]{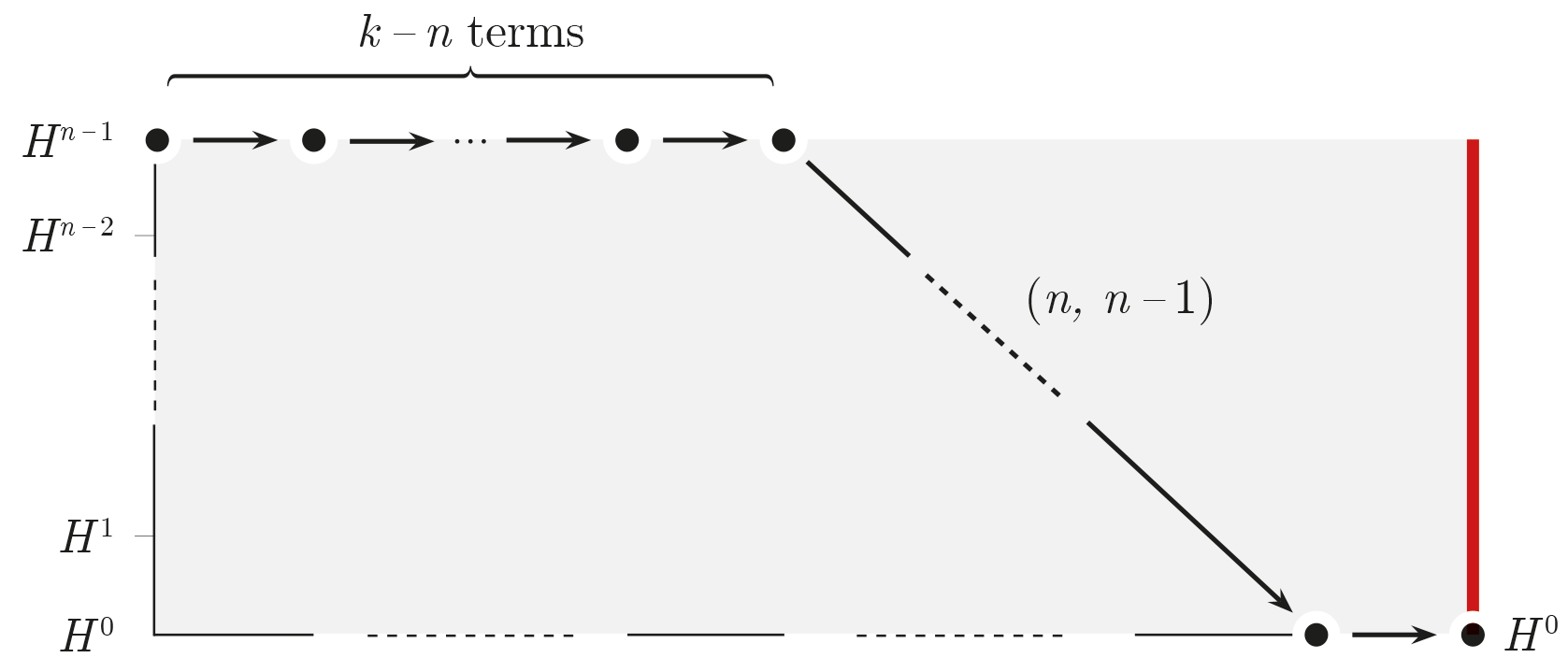}\\
	\end{center}
	\medskip \par
According to Leray's theorem, the spectral sequence for the exact sequence converges to zero. The only term in the first column that can be cancelled by the other terms in the spectral sequence is the term in the $0$-th line. This means that $H^i \left(\mathbb{P}V_a, \operatorname{Sym}^l\left( {{W\otimes \mathbb{C}^k}\over{L_2\otimes \mathbb{C}^k}}\right)\right)$ vanishes for $i>0.$\par 
Applying the pushforward $(p_2)_*,$ we obtain $$H^i (Y,\mathcal{O}_Y)=H^i \left(\mathbb{P}^{n-1}, H^0\left(\mathbb{P}V_a,\operatorname{Sym}^l\left( {{W\otimes \mathbb{C}^k}\over{L_2\otimes \mathbb{C}^k}}\right)\right)\right).$$\par
Let us construct the resolution of $H^0\left(\operatorname{Sym}^l\left( {{W\otimes \mathbb{C}^k}\over{L_2\otimes \mathbb{C}^k}}\right)\right).$ In the spectral sequence above, whatever remains in the line number $n-1$ after the first page goes exactly to $\operatorname{Sym}^l(W\otimes \mathbb{C}^k)$ in the line number $0$ on the $n$-th page. This allows us to write down the following resolution:
$$\operatorname{det}V_a\otimes \operatorname{Sym}^{k-n}V_a\otimes\Lambda^k\mathbb{C}^k\otimes \operatorname{Sym}^{l-k}(W\otimes \mathbb{C}^k)\rightarrow \mydots$$ $$\mydots \rightarrow  \operatorname{det}V_a\otimes \operatorname{Sym}^{k-n-i}V\otimes\Lambda^{k-i}\mathbb{C}^k\otimes \operatorname{Sym}^{l-(k-i)}(W\otimes \mathbb{C}^k) \rightarrow\mydots $$ $$\mydots \rightarrow \operatorname{det}V_a\otimes \Lambda^n\mathbb{C}^k\otimes \operatorname{Sym}^{l-n}(W\otimes \mathbb{C}^k)\rightarrow \operatorname{Sym}^l(W\otimes \mathbb{C}^k) \rightarrow H^0\left(\operatorname{Sym}^l\left( {{W\otimes \mathbb{C}^k}\over{L_2\otimes \mathbb{C}^k}}\right)\right)$$\par
Which can be presented in the following form.\par
\textbf{Resolution 2:}
\begin{equation}
\begin{gathered}
\operatorname{det}Q\otimes L_1^2\otimes \operatorname{Sym}^{k-n}(Q_1\oplus L_1^2)\otimes\Lambda^k\mathbb{C}^k\otimes \operatorname{Sym}^{l-k}(({\operatorname{Sym}^2\mathbb{C}^n\oplus Q_1})\otimes \mathbb{C}^k)\rightarrow \mydots  \nonumber \\ \mydots \rightarrow  \operatorname{det}Q\otimes L_1^2\otimes \operatorname{Sym}^{k-n-i}(Q_1\oplus L_1^2)\otimes\Lambda^{k-i}\mathbb{C}^k\otimes \operatorname{Sym}^{l-(k-i)}(({\operatorname{Sym}^2\mathbb{C}^n\oplus Q_1})\otimes \mathbb{C}^k) \rightarrow \mydots  \\ \mydots \rightarrow \operatorname{det}Q\otimes L_1^2\otimes\Lambda^n\mathbb{C}^k\otimes \operatorname{Sym}^{l-n}(({\operatorname{Sym}^2\mathbb{C}^n\oplus Q_1})\otimes \mathbb{C}^k)\rightarrow  \\ \rightarrow \operatorname{Sym}^l(({\operatorname{Sym}^2\mathbb{C}^n\oplus Q_1})\otimes \mathbb{C}^k) \rightarrow H^0\left(\operatorname{Sym}^l\left( {{({\operatorname{Sym}^2\mathbb{C}^n\oplus Q_1})\otimes \mathbb{C}^k}\over{L_2\otimes \mathbb{C}^k}}\right)\right)
\end{gathered}\label{res1}
\end{equation}
This allows us to formulate our first result.
\begin{theorem}\label{rat}
$\widetilde{\Theta_{A_2}^{n,n}}$ has rational singularities. 	
\end{theorem}
\begin{proof}
	If $k=n$ then Resolution 2 may be rewritten as follows:
	\begin{equation}
\begin{gathered} 
\operatorname{det}Q\otimes L_1^2\otimes\Lambda^k\mathbb{C}^k\otimes \operatorname{Sym}^{l-k}((\operatorname{Sym}^2\mathbb{C}^k\oplus Q_1)\otimes \mathbb{C}^k)\rightarrow \nonumber \\
\rightarrow \operatorname{Sym}^l((\operatorname{Sym}^2\mathbb{C}^k\oplus Q_1)\otimes \mathbb{C}^k)\rightarrow \nonumber \tag{$\star$} \\
  \rightarrow H^0\left(\operatorname{Sym}^l\left( {{({\operatorname{Sym}^2\mathbb{C}^k\oplus Q_1})\otimes \mathbb{C}^k}\over{L_2\otimes \mathbb{C}^k}}\right)\right) \nonumber
\end{gathered}\label{res}
\end{equation}	
We will prove that, in the corresponding spectral sequence, there are no non-trivial terms above the $0$-th line.
\begin{lemma}
	$$\operatorname{Sym}^N((\operatorname{Sym}^2\mathbb{C}^k\oplus Q_1)\otimes \mathbb{C}^k)=$$ 
$$=\bigoplus_{i=0}^N \left( \left(\operatorname{Sym}^{N-i}\left(\operatorname{Sym}^2\mathbb{C}^k\otimes \mathbb{C}^k\right) \right) \otimes \bigoplus_{(i_1,\dots,i_k)}^{i_1+\dots+i_k=i}\operatorname{Sym}^{i_1}Q_1\otimes \dots \otimes \operatorname{Sym}^{i_k}Q_1\right).$$
\end{lemma}
Setting $N=l,$ the lemma provides the decomposition of $\operatorname{Sym}^l((\operatorname{Sym}^2\mathbb{C}^k\oplus Q_1)\otimes \mathbb{C}^k).$ The only non-constant sheaves here are the sheaves of the form $$\operatorname{Sym}^{i_1}Q_1\otimes \dots \otimes \operatorname{Sym}^{i_k}Q_1.$$
We decompose this tensor product into a sum of irreducible representations:$$\operatorname{Sym}^{i_1}Q_1\otimes \dots \otimes \operatorname{Sym}^{i_m}Q_1=\bigoplus_{\lambda} a_{\lambda} \Sigma^{\lambda}Q_1,$$
where $\lambda=(\lambda_1,\mydots,\lambda_n),\ \sum \lambda_k=\sum i_j,$ and $a_{\lambda}$ are non-negative integers.\par
 Since there is no multiplication by a power of $L_1$ and $\lambda$ is already dominant, i.e. strictly decreasing, by the Borel-Weil-Bott theorem $H^i(\mathbb{P}^{n-1},\operatorname{Sym}^{i_1}Q_1\otimes \dots \otimes \operatorname{Sym}^{i_k}Q_1)=0$ for $i>0$.\par
This proves that the term in the second line of the resolution ~\eqref{res}  does not have any higher cohomology.\par
However, the term in the first line of the resolution ~\eqref{res} has $L_1^2$ as a multiplier. As before, we use the lemma above for $N=l-k$ to find the decomposition of this term. The non-trivial part in this case is the following:
$$\operatorname{det} Q_1\otimes L_1^2\otimes \bigoplus_{\lambda} a_{\lambda}\Sigma^{\lambda}Q_1=\operatorname{det} \mathbb{C}^n\otimes L_1\otimes \bigoplus_{\lambda} a_{\lambda}\Sigma^{\lambda}Q_1.$$
Let us apply the Borel-Weil-Bott theorem to $\Sigma^{\lambda}Q_1\otimes L_1:$  
$$(\lambda_1,\mydots,\lambda_{n-1},1)+(n,\mydots,1)=(\nu_1+n,\mydots,\nu_{n-1}+2,2).$$
Since $\nu_{n-1}\geq 0,$ we either have a dominant sequence if $\nu_{n-1}>0,$ or a repetition if $\nu_{n-1}=0.$ In both cases there is no higher cohomology.\par
So, there are no non-trivial entries in the corresponding Leray spectral sequence above the $0$-th line, so $H^i(Y,\mathcal{O}_Y)=0$ for $i>0,$ and $\widetilde{\Theta_{A_1}^{n,n}}$ has rational singularities.
\end{proof}
\begin{theorem} \label{notrat}
$\widetilde{\Theta_{A_2}^{n,k}}$ in general has singularities worse than rational.	
\end{theorem}
\begin{proof}
	Consider the case $n=5,\ k=7,\ l=7.$ \par
	We prove that $H^1\left( \mathbb{P}^4,\ \operatorname{Sym}^{7} \left( {{\left(\operatorname{Sym}^2\mathbb{C}^5\oplus Q_1\right)\otimes \mathbb{C}^7}\over{L_2\otimes \mathbb{C}^7}}\right)\right)\not\cong 0.$
	In this particular case Resolution 2 is the following:
	$$\opn{det} Q_1\otimes L_1^2\otimes \operatorname{Sym}^2(Q_1\oplus L_1^2)\rightarrow$$ 
	$$\rightarrow \opn{det} Q_1\otimes L_1^2\otimes (Q_1\oplus L_1^2)\otimes \Lambda^6\mathbb{C}^7\otimes((\operatorname{Sym}^2\mathbb{C}^5\oplus Q_1)\otimes \mathbb{C}^7)\rightarrow$$
	$$\rightarrow \opn{det} Q_1\otimes L_1^2\otimes \Lambda^5\mathbb{C}^7\otimes\operatorname{Sym}^2((\operatorname{Sym}^2\mathbb{C}^5\oplus Q_1)\otimes \mathbb{C}^7)\rightarrow$$
	$$\rightarrow \operatorname{Sym}^{7}((\operatorname{Sym}^2\mathbb{C}^5\oplus Q_1)\otimes \mathbb{C}^7)\rightarrow$$
	$$\rightarrow H^0\left( \operatorname{Sym}^{7} \left( {{\left(\operatorname{Sym}^2\mathbb{C}^5\oplus Q_1\right)\otimes \mathbb{C}^7}\over{L_2\otimes \mathbb{C}^7}}\right)\right)$$\par
	Consider the term in the first line of the resolution above.
	$$\opn{det} Q_1\otimes L_1^2\otimes \operatorname{Sym}^2(Q_1\oplus L_1^2)= \opn{det} Q_1\otimes L_1^2\otimes \left( \opn{Sym}^2 Q_1 \oplus Q_1\otimes L_1^2 \oplus L_1^4\right)=$$
	$$=\opn{det}Q_1\otimes L_1^6 \oplus \opn{det}Q_1 \otimes L_1^2 \left(\opn{Sym}^2 Q_1 \oplus Q_1\otimes L_1^2\right).$$
	Using the Borel-Weil-Bott theorem, one can easily check that $$H^4 \left( \mathbb{P}^4,\ \opn{det}Q_1\otimes L_1^6 \right) \not \cong 0,$$ $$H^0 (\mathbb{P}^4,\ \operatorname{Sym}^{7}((\operatorname{Sym}^2\mathbb{C}^5\oplus Q_1)\otimes \mathbb{C}^7))\not \cong 0,$$ but all other terms of the resolution do not have any cohomology.\par
	 The corresponding Leray spectral sequence is the following:\\
	\begin{center}
	\includegraphics[width=\textwidth]{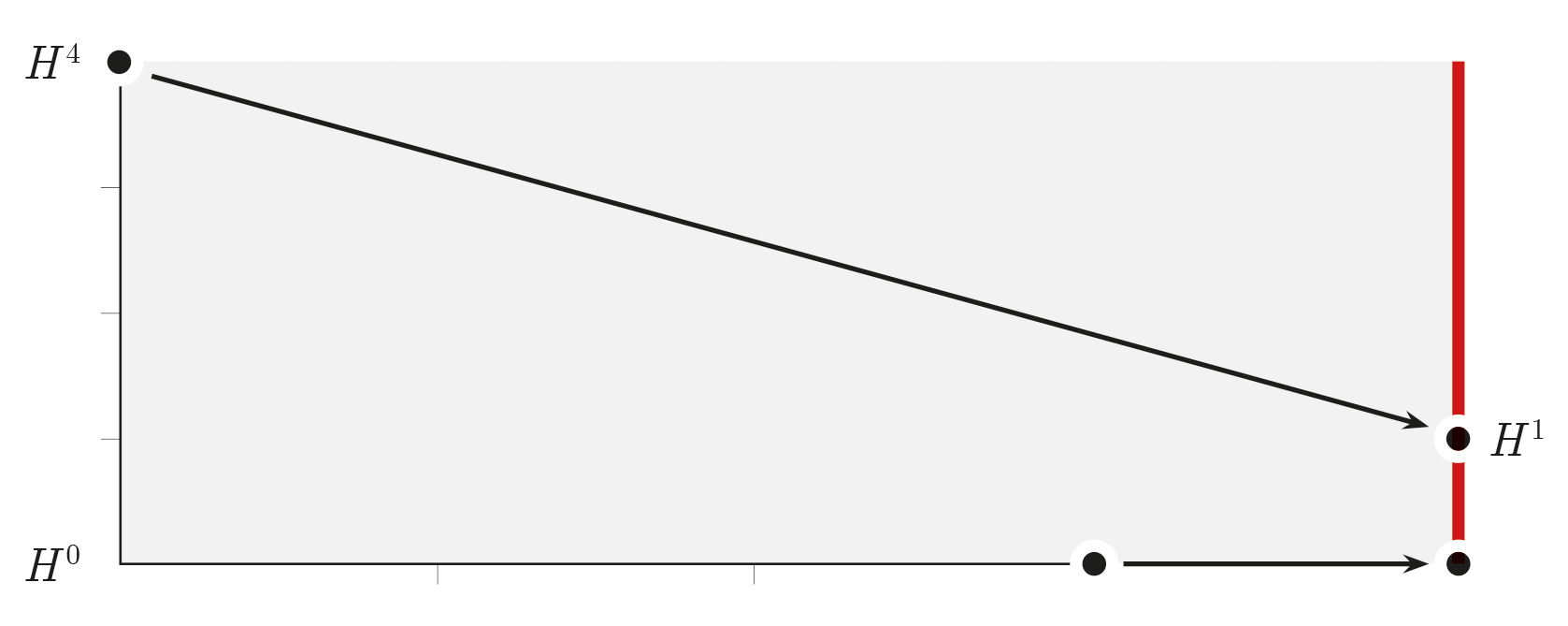}\\
	\end{center}
	\medskip \par
	Thus, we proved that $$H^1\left(\mathbb{P}^4,\ \operatorname{Sym}^{7} \left( {{\left(\operatorname{Sym}^2\mathbb{C}^5\oplus Q_1\right)\otimes \mathbb{C}^7}\over{L_2\otimes \mathbb{C}^7}}\right)\right) \not \cong 0,$$
	and therefore $\widetilde{\Theta_{A_2}^{5,7}}$ has singularities worse than rational.
\end{proof}
According to Boutot \cite{Bou}, the GIT quotient of a smooth variety with respect to a reductive group has rational singularities. Thus, we have the following corollary of the Theorem \ref{notrat}.
\begin{corollary}
$\Theta_{A_2}^{n,k}$ can not be presented as a reductive quotient of a smooth variety.	
\end{corollary}
For the recent results on the GIT quotient with respect to non-reductive groups, see the works of Kirwan and B\'erczi \cite{KirB}, and  B\'erczi,  Doran, Hawes and Kirwan \cite{Ber}.
\begin{remark}
In both Theorem \ref{rat} and Theorem \ref{notrat} we consider the normalizations of the $A_2$-loci. Let us show that the normalization is not redundant, i.e. that $\Theta_{A_2}^{n,k}$ is not always normal.\par
Let $V$ be a complex vector space equipped with the action of a compact Lie group $G,$ and let $X$ be a closed $G$-invariant subvariety of $V$. Suppose $Y$ is a smooth $G$-equivariant resolution of $X.$\par
Consider the following diagram:
\begin{center}
\begin{tikzpicture}
  \matrix (m) [matrix of math nodes,row sep=3em,column sep=4em,minimum width=2em]
  {
    &  H^0 (Y,\mathcal{O}_Y)\\
     \hspace{-6em} H^0(V,\mathcal{O}_V)=\bigoplus_l\opn{Sym}^l V^* \\
      & H^0(X,\mathcal{O}_x)\\};
  \path[-stealth]
    (m-2-1) edge node[above] {$f$} (m-1-2)
    (m-2-1) edge node [below]{$g$} (m-3-2)
    (m-3-2) edge node [right ] {$h$} (m-1-2)
            ;
\end{tikzpicture}\\
\end{center}\par
We know that $g$ is always surjective, and, according to Proposition \ref{pr}, $h$ is an isomorphism if and only if $X$ is normal. Now, if $f$ is not surjective, then $h$ can not be an isomorphism, and therefore in this case $X$ is not a normal variety.\par 
Let $V=J_2 (n,k),$ $G=\opn{Gl}(n)\times \opn{Gl}(k),$ $X=\Theta_{A_2}^{n,k},$ and let $Y$ be the Kazarian's smooth equivariant resolution of $\Theta_{A_2}^{n,k}.$\par
Consider Resolution 2 in the general case:
\begin{equation}
\begin{gathered}
\operatorname{det}Q\otimes L_1^2\otimes \operatorname{Sym}^{k-n}(Q_1\oplus L_1^2)\otimes\Lambda^k\mathbb{C}^k\otimes \operatorname{Sym}^{l-k}(({\operatorname{Sym}^2\mathbb{C}^n\oplus Q_1})\otimes \mathbb{C}^k)\rightarrow \mydots  \nonumber \\ \mydots \rightarrow  \operatorname{det}Q\otimes L_1^2\otimes \operatorname{Sym}^{k-n-i}(Q_1\oplus L_1^2)\otimes\Lambda^{k-i}\mathbb{C}^k\otimes \operatorname{Sym}^{l-(k-i)}(({\operatorname{Sym}^2\mathbb{C}^n\oplus Q_1})\otimes \mathbb{C}^k) \rightarrow \mydots  \\ \mydots \rightarrow \operatorname{det}Q\otimes L_1^2\otimes\Lambda^n\mathbb{C}^k\otimes \operatorname{Sym}^{l-n}(({\operatorname{Sym}^2\mathbb{C}^n\oplus Q_1})\otimes \mathbb{C}^k)\rightarrow  \\ \rightarrow \operatorname{Sym}^l(({\operatorname{Sym}^2\mathbb{C}^n\oplus Q_1})\otimes \mathbb{C}^k) \rightarrow H^0\left(\mathbb{P}^{n-1},\ \operatorname{Sym}^l\left( {{({\operatorname{Sym}^2\mathbb{C}^n\oplus Q_1})\otimes \mathbb{C}^k}\over{L_2\otimes \mathbb{C}^k}}\right)\right).
\end{gathered}
\end{equation}\par
Recall that $$H^0(Y,\mathcal{O}_Y)=\bigoplus_l H^0\left(\mathbb{P}^{n-1},\ \operatorname{Sym}^l\left( {{({\operatorname{Sym}^2\mathbb{C}^n\oplus Q_1})\otimes \mathbb{C}^k}\over{L_2\otimes \mathbb{C}^k}}\right)\right) \text{ and }$$ $$H^0(V,\mathcal{O}_V)= \bigoplus_l \operatorname{Sym}^l(({\operatorname{Sym}^2\mathbb{C}^n\oplus \C^n})\otimes \mathbb{C}^k)=\bigoplus_l H^0(\mathbb{P}^{n-1},\operatorname{Sym}^l(({\operatorname{Sym}^2\mathbb{C}^n\oplus Q_1})\otimes \mathbb{C}^k)).$$\par
Since the map $f$ from the diagram above preserves the graded components, it is enough to prove that $$f_l\colon \operatorname{Sym}^l(({\operatorname{Sym}^2\mathbb{C}^n\oplus \C^n})\otimes \mathbb{C}^k) \longrightarrow H^0\left(\mathbb{P}^{n-1},\ \operatorname{Sym}^l\left( {{({\operatorname{Sym}^2\mathbb{C}^n\oplus Q_1})\otimes \mathbb{C}^k}\over{L_2\otimes \mathbb{C}^k}}\right)\right) $$
is not surjective for some fixed $l.$\par
Note that $f_l$ is the right arrow in the line $H^0$ of the first page of the Leray spectral sequence corresponding to Resolution 2. That is, if we can find an example of a spectral sequence with a non-horizontal arrow pointing to the term $H^0\left(\mathbb{P}^{n-1},\ \operatorname{Sym}^l\left( {{({\operatorname{Sym}^2\mathbb{C}^n\oplus Q_1})\otimes \mathbb{C}^k}\over{L_2\otimes \mathbb{C}^k}}\right)\right),$ we prove that $f$ is not surjective. \par
Let $n=3,\ k=4,\ l=4.$ In this case Resolution 2 is the following:
\begin{equation}
\begin{gathered}
 \operatorname{det}Q\otimes L_1^2\otimes (Q_1\oplus L_1^2)\otimes\Lambda^{4}\mathbb{C}^4  \rightarrow  \nonumber \\ 
 \rightarrow \operatorname{det}Q\otimes L_1^2\otimes\Lambda^3\mathbb{C}^4\otimes (({\operatorname{Sym}^2\mathbb{C}^3\oplus Q_1})\otimes \mathbb{C}^4)\rightarrow  \\ \rightarrow \operatorname{Sym}^4(({\operatorname{Sym}^2\mathbb{C}^3\oplus Q_1})\otimes \mathbb{C}^4) \rightarrow
  H^0\left(\mathbb{P}^{2},\ \operatorname{Sym}^4\left( {{({\operatorname{Sym}^2\mathbb{C}^3\oplus Q_1})\otimes \mathbb{C}^4}\over{L_2\otimes \mathbb{C}^4}}\right)\right).
\end{gathered}
\end{equation}\par
A straightforward computation using the Borel-Weil-Bott theorem shows that the corresponding Leray spectral sequence is the following.\par \medskip
\begin{center}
	\includegraphics[width=\textwidth]{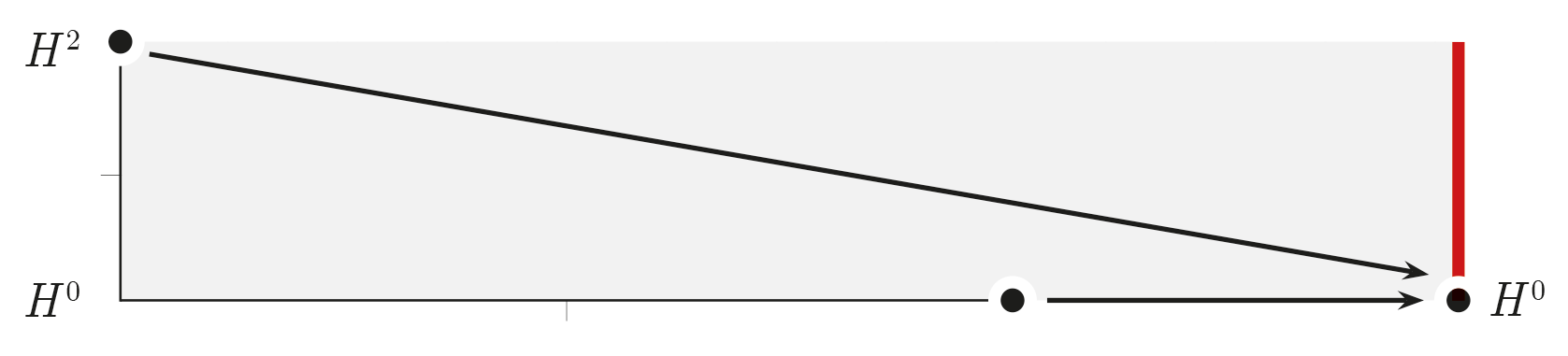}\\
	\end{center}
	\medskip \par
We see that there is a non-horizontal arrow pointing to $H^0\left(\mathbb{P}^{2},\ \operatorname{Sym}^4\left( {{({\operatorname{Sym}^2\mathbb{C}^3\oplus Q_1})\otimes \mathbb{C}^4}\over{L_2\otimes \mathbb{C}^4}}\right)\right),$ thus $\Theta_{A_2}^{3,4}$ is not a normal variety.
\end{remark}
\begin{remark}
Since the equivariant resolutions for the $A_3$-loci given in \cite{BSz} and \cite{Kaz} are smooth, the computational methods presented in this paper may be used to check the rationality of the singularities of $\Theta_{A_3}^{n,k}.$
\end{remark}

 \end{document}